\documentclass[11pt, a4paper]{amsart}
\usepackage{amsmath, amsthm, amsfonts, amssymb, mathtools}
\usepackage{a4wide}
\usepackage[latin1]{inputenc}

\newcommand{\SL}{\mathrm{SL}}
\newcommand{\Mp}{\mathrm{Mp}}

\newcommand{\N}{\mathbb{N}}
\newcommand{\Z}{\mathbb{Z}}
\newcommand{\R}{\mathbb{R}}
\newcommand{\C}{\mathbb{C}}
\newcommand{\Q}{\mathbb{Q}}
\renewcommand{\H}{\mathcal{H}}
\DeclareMathOperator{\tr}{tr}

\DeclareMathOperator{\sgn}{sgn}

\DeclareMathOperator{\e}{\mathfrak{e}}

\renewcommand{\Im}{\mathrm{Im}}

\numberwithin{equation}{section}
	\newtheorem{Satz}{Satz}[section]
	\newtheorem{theorem}[Satz]{Theorem}
	\newtheorem{lemma}[Satz]{Lemma}
	\newtheorem{proposition}[Satz]{Proposition}

	\theoremstyle{definition} 
	
	\newtheorem{example}[Satz]{Example}
	\newtheorem{remark}[Satz]{Remark}
\date{\today}

\author{Markus Schwagenscheidt and Brandon Williams}

\address{ETH Z\"urich Mathematics Dept., R\"amistrasse 101, CH-8092 Z\"urich, Switzerland}
\email{mschwagen@ethz.ch}
\address{Lehrstuhl A f\"ur Mathematik, RWTH Aachen University, 52056 Aachen, Germany}
\email{brandon.williams@matha.rwth-aachen.de}

\title{Binary theta functions and Borcherds products}

\begin{document}

\maketitle

\begin{abstract}
	We obtain infinite product expansions in the sense of Borcherds for theta functions associated with certain positive-definite binary quadratic and binary hermitian forms. Among other things, we show that every weight $1$ binary theta function is a Borcherds product. In particular, binary theta functions have zeros only at quadratic irrationalities. 
\end{abstract}

\section{Introduction}
Infinite products appear in many places in number theory and combinatorics, for instance in the Euler product of the Riemann zeta function, the Jacobi triple product, or the generating series for the partition function. A prominent example from the theory of modular forms is Ramanujan's $\Delta$-function
\[
\Delta(\tau) = q\prod_{n=1}^\infty (1-q^n)^{24} \qquad \left(q=e^{2\pi i \tau}\right),
\]
which defines a cusp form of weight $12$ for $\SL_2(\Z)$. The modularity of this and many similar infinite products can be systematically explained (and vastly generalized) using Borcherds' theory of automorphic products \cite{borcherds95, borcherds}. In the simplest case, this theory associates to a weakly holomorphic modular form $f = \sum_{n \gg -\infty}c_f(n)q^n$ of weight $1/2$ for $\Gamma_0(4)$ (with integral coefficients $c_f(n)$ satisfying the Kohnen plus space condition $c_f(n) = 0$ unless $n \equiv 0,1 \pmod 4$) the infinite product
\begin{align}\label{Borcherds product}
\Psi_f(\tau) = q^{\varrho_f} \prod_{n=1}^{\infty}(1-q^n)^{c_f(n^2)},
\end{align}
where $\varrho_f \in \Q$ is the so-called Weyl vector, which can also be computed from $f$. The infinite product converges for $\Im(\tau)$ large enough and extends to a meromorphic modular form of weight $c_f(0)$ for $\SL_2(\Z)$ with some character. Moreover, the roots and poles of $\Psi_f(\tau)$ lie at certain imaginary quadratic irrationalities which are determined by the principal part of $f$. For example, $\Delta(\tau)$ is the Borcherds product associated to the Jacobi theta function $12\theta(\tau) = 12\sum_{n \in \Z}q^{n^2}$, and the fact that $12\theta(\tau)$ has vanishing principal part reflects the fact that $\Delta(\tau)$ does not have any roots (or poles) on the complex upper half-plane $\H$. Similarly, the modular $j$-invariant and the Eisenstein series $E_k$ for $k = 4,6,8,10, 14$ are Borcherds products, see \cite{borcherds95}.

More generally, the multiplicative Borcherds lift is a construction of automorphic products for orthogonal groups of signature $(2,n)$ for any $n \geq 1$, see \cite[Theorem~13.3]{borcherds}. As special cases, this includes elliptic modular forms for congruence subgroups, Hilbert modular forms over real quadratic fields, and Siegel modular forms of genus $2$.

In the present note, we obtain infinite product expansions for the holomorphic theta functions associated with certain positive definite, integral binary quadratic and binary hermitian forms. 
We will discuss weight $1$ theta functions associated with integral binary quadratic forms, weight $2$ theta functions associated with binary hermitian forms over the Gaussian integers $\Z[i]$, and weight $4$ theta functions associated with binary hermitian forms over the Lipschitz integers $\Z[i,j,k]$ (a subring of the quaternions $\mathbb{H}$).

We illustrate our results in the case of binary quadratic forms in the introduction. Let 
\[
A(x,y) = ax^2 + bxy + cy^2, \qquad a,b,c \in \Z,
\]
be a positive definite, integral binary quadratic form of discriminant $D = b^2-4ac < 0$, and let
\[
\vartheta_A(\tau) = \sum_{m,n \in \mathbb{Z}} q^{A(m,n)}
\]
be the corresponding binary theta function. It is a modular form of weight $1$ for $\Gamma_0(|D|)$ and character $\chi_D = \left(\frac{D}{\cdot}\right)$, and it is an eigenform of the Fricke involution with eigenvalue $-i$ (compare \cite[\S~4.9]{miyake}), i.e. $$\vartheta_A\left(-\frac{1}{|D|\tau}\right) = -i \tau \sqrt{|D|} \cdot \vartheta_A(\tau).$$

\begin{theorem}\label{theorem product expansion intro}
	For $\Im(\tau)$ large enough we have the product expansion
	\begin{align}\label{product expansion}
	\vartheta_A(\tau) = \prod_{n=1}^\infty\left(\frac{1+q^n}{1-q^n}\right)^{d_A(n)},
	\end{align}
	with
	\[
	d_A(n) = \!\!\!\!\sum_{\substack{\alpha, \beta, \gamma \in \Z \\ \alpha a + \beta b + \gamma c  = n}}\!\!\!\!(-1)^{(\alpha+1)(\gamma+1)}c_{1/\theta}(\alpha\gamma-\beta^2) \in \Z.
	\]
	Here we wrote $1/\theta(\tau) = \sum_{n\geq 0}c_{1/\theta}(n)q^n = 1 - 2q + 4q^2 - 8q^3+\dots$ with the Jacobi theta function $\theta(\tau) = \sum_{n \in \Z}q^{n^2}$.
\end{theorem}

\begin{remark}\label{remark finite}
	For fixed $n \in \N$, the number of $\alpha,\beta,\gamma \in \Z$ with $n = \alpha a + \beta b + \gamma c$ with $\alpha \gamma -\beta^2 \geq 0$ is finite. A simple bound is given by $0 \leq \alpha \leq \frac{4n}{|D|}c, ~ 0 \leq \gamma \leq \frac{4n}{|D|}a,$ and  $|\beta| \leq \sqrt{\alpha \gamma}$.
\end{remark}

\begin{example}\label{example root lattice A2}
	We consider the quadratic form $A = [1,1,1]$ of discriminant $D = -3$ and the corresponding theta function
	\[
	\vartheta_{[1,1,1]}= 1 + 6q +6q^3 +6q^4 + 12q^7 + 6q^9 + 6q^{12} + 12q^{13} + \ldots \in M_{1}(\Gamma_0(3),\chi_{-3}).
	\]
	Using Remark~\ref{remark finite} we may compute the exponents $d_{[1,1,1]}(n)$ in Theorem~\ref{theorem product expansion intro}, and obtain the product expansion
	\[
	\vartheta_{[1,1,1]} = \left( \frac{1+q}{1-q}\right)^3\left( \frac{1+q^2}{1-q^2}\right)^{-9}\left( \frac{1+q^3}{1-q^3}\right)^{38}\left( \frac{1+q^4}{1-q^4}\right)^{-177}\left( \frac{1+q^5}{1-q^5}\right)^{867}\cdots
	\]
\end{example}

The proof of Theorem~\ref{theorem product expansion intro} is quite short if we use some well-known results about Siegel modular forms of genus $2$. The key observation, due to C.F.~Hermann \cite{hermann}, is the fact that the binary theta function $\vartheta_A(\tau)$ is the restriction to $\{Z = \big( \begin{smallmatrix}2a\tau & b\tau \\ b\tau & 2c\tau \end{smallmatrix}\big) \, : \, \tau \in \H\}$ of the genus $2$ Siegel theta constant
\[
\Theta(Z) = \sum_{m,n \in \Z}q^{m^2/2} r^{mn}s^{n^2/2},
\]
where we wrote $Z = \left(\begin{smallmatrix}\tau & z \\ z & w\end{smallmatrix}\right) \in \H_2$ (the Siegel upper half-space of genus $2$) and $q = e^{2\pi i \tau}, r = e^{2\pi i z}, s = e^{2\pi i w}$. It is known that $\Theta(Z)$ is a weight $1/2$ Siegel modular form of genus $2$, and has an infinite product expansion. More precisely, Borcherds \cite[Example~5 in Section~15]{borcherds95} showed that\footnote{There is a sign error in \cite{borcherds95} that we have corrected here.}
\[
\Theta(Z)  = \prod_{\substack{\alpha,\beta,\gamma \in \Z \\ \alpha+\gamma > 0}}\left(\frac{1+(-1)^{(\alpha+1)(\gamma+1)}q^{\alpha/2} r^{\beta} s^{\gamma/2}}{1-(-1)^{(\alpha+1)(\gamma+1)}q^{\alpha/2} r^{\beta} s^{\gamma/2}}\right)^{c_{1/\theta}(\alpha \gamma -\beta^2)}.
\]
The condition $\alpha + \gamma > 0$ may be replaced by $\alpha a + \beta b + \gamma c > 0$, see \cite[Lemma~3.2]{bruinierhabil}. Plugging in $Z = \left( \begin{smallmatrix}2a\tau & b\tau \\ b\tau & 2c\tau \end{smallmatrix}\right)$, we obtain the product expansion for $\vartheta_A(\tau)$ as stated in Theorem~\ref{theorem product expansion intro}. It would be interesting to find a more direct (e.g. combinatorial) proof of the product expansion for $\vartheta_A(\tau)$, without the detour via Siegel modular forms.
 
By computing the Borcherds product expansion of the Siegel theta constant $\Theta(Z)$ at a different cusp, we obtain another product formula for $\vartheta_A(\tau)$, more reminiscent of the infinite product in \eqref{Borcherds product}.

\begin{theorem}\label{theorem Borcherds products intro}
	Suppose that $A$ has odd fundamental discriminant $D < 0$. Then the binary theta function $\vartheta_A(\tau)$ is a Borcherds product associated with a weakly holomorphic modular form $F_A = \sum_{n \gg -\infty}c_A(n)q^n$ of weight $1/2$ for $\Gamma_0(4|D|)$. In particular, for $\Im(\tau)$ large enough we have an infinite product expansion 
	\[
	\vartheta_A(\tau) = \prod_{n=1}^\infty (1-q^n)^{\tilde{c}_A(n^2)},
	\]
	where $\tilde{c}_A(n) = c_A(n)$ if $N \mid n$ and $\tilde{c}_A(n) = c_A(n)/2$ if $N \nmid n$. The coefficients $\tilde{c}_A(n)$ are integers.
\end{theorem}

Note that the Fourier expansion of $F_A$ can be described explicitly, see Remark~\ref{remark Borcherds products intro} below. In Theorem~\ref{theorem Borcherds products} we will prove a more general version of Theorem~\ref{theorem Borcherds products intro} for arbitrary discriminants $D$ by describing the Borcherds input $F_A$ as a vector-valued modular form for the Weil representation for the lattice $A_1(|D|) = (\Z,|D|x^2)$. This form is essentially constructed by ``multiplying'' a vector-valued version of $1/\theta$ with a certain vector-valued binary theta function. 

\begin{remark}\label{remark Borcherds products intro}
	The Fourier expansion of the Borcherds input $F_A$ from Theorem~\ref{theorem Borcherds products intro} can be described in terms of the Fourier coefficients of the weight $-1/2$ modular forms $1/\theta$ and $\eta(2\tau)/\eta(4\tau)^{2} = q^{-1/4}(1-q^{2} + q^{4} - 2q^{6} +  \dots)$. More precisely, we have the explicit formula
	\[
	F_A = \sum_{\ell \in \Z}\sum_{h = 1}^{2|D|}\sum_{\substack{\alpha,\beta,\gamma \in \Z \\ \alpha a +\beta  b + \gamma c = h }}c(\alpha,\beta,\gamma,\ell)q^{4|D|\ell+h^2}
	\]
	where 
	\begin{align*}
	c(\alpha,\beta,\gamma,\ell) = \begin{dcases}
	-\frac{2}{\theta(\tau/4)}\left[\ell+\tfrac{1}{4}(\alpha \gamma-\beta^2)\right], & \text{$\alpha$ or $\gamma$ odd,} \\
	\begin{multlined}
		 \left(\frac{2}{\theta(\tau/4)}+(-1)^{(\frac{\alpha}{2}+1)(\frac{\gamma}{2}+1)}\frac{1}{\theta(\tau)}\right.  \\ 
		 \left. \qquad + \frac{1}{2}\left(1-(-1)^{\frac{\alpha\gamma}{4}} \right)\frac{\eta(2\tau)}{\eta(4\tau)^2} \right)\left[\ell+\tfrac{1}{4}(\alpha \gamma-\beta^2)\right],\end{multlined}& \text{$\alpha,\gamma$ even.}
	\end{dcases}
	\end{align*}
	Here we wrote $f[\ell]$ for the coefficient at $q^\ell$ in a $q$-series $f$. 
\end{remark}

\begin{example}
	We again consider the quadratic form $A = [1,1,1]$. The Borcherds input $F_A$ from Theorem~\ref{theorem Borcherds products intro} can be computed using Remark~\ref{remark Borcherds products intro} as
	\begin{align*}
	F_{[1,1,1]} = q^{-3} &+ 1 - 12q + 42q^4 - 76q^9 + 168q^{12} - 378q^{13}  \\
	&+ 690q^{16} - 897q^{21} + 1456q^{24} - 3468q^{25} + \ldots \in M_{1/2}^!(12).
	\end{align*}
	We obtain the product expansion
	\[
	\vartheta_{[1,1,1]} = (1-q)^{-12/2}(1-q^2)^{42/2}(1-q^3)^{-76}(1-q^4)^{690/2}(1-q^5)^{-3468/2}\cdots
	\]
\end{example}

An important property of Borcherds products is the fact that their roots and poles lie at quadratic irrationalities, whose locations and orders can be read off from the principal part of the Borcherds input. However, instead of analyzing the principal part of $F_A$, it seems to be slightly easier to read off the roots of $\vartheta_A$ directly from the fact that $\vartheta_A$ is a restriction of the degree $2$ theta constant $\Theta(Z)$, whose roots are well-known. This was carried out in detail by Hermann in \cite{hermann}. To describe the roots, we let $\mathcal{Q}_{N,d,h}$ be the set of all level $N$ CM (or Heegner) points of discriminant $d$ and residue class $h$, which are defined to be the roots $\tau \in \H$ of quadratic polynomials $A\tau^2 + B \tau + C$ for some $A,B,C \in \Z$ with $A > 0, A \equiv 0 \pmod N$, $B \equiv h \pmod{2N}$, and $B^2-4AC = d < 0$.

\begin{theorem}[Hermann \cite{hermann}]
	Let $A$ be a positive definite, integral binary quadratic form of discriminant $D <0$. Then $\vartheta_A$ has simple roots precisely at the level $|D|$ CM points in the union of the sets $\mathcal{Q}_{|D|,d,h}$ over all $0 \leq h < 2|D|$ and $D < d < 0$ of the form
	\begin{align*}
		h &= \alpha a + \beta b + \gamma c, \qquad	d = D + D(\alpha \gamma - \beta^2) + h^2,
	\end{align*}
	for some $\alpha,\beta,\gamma \in \Z$ satisfying $\alpha \gamma - \beta^2 > 0$ and $\alpha,\gamma \equiv 2 \pmod 4, \beta$ odd.
\end{theorem}

\begin{remark}
	The above description yields a simple algorithm to determine all roots of $\vartheta_A$.  First, find all $\alpha,\beta,\gamma \in \Z$ with $\alpha \gamma - \beta^2 > 0$ and $\alpha,\gamma \equiv 2 \pmod 4,\beta$ odd, such that $h = \alpha a + \beta b + \gamma c$ satisfies $0 \leq h < 2|D|$. By Remark~\ref{remark finite}, we only need to search in the range $0 < \alpha \leq 8c, 0<\gamma \leq 8a, |\beta| < \sqrt{\alpha \gamma}$. Then, for every such tuple $(\alpha,\beta,\gamma)$, check whether $d = D + D(\alpha \gamma - \beta^2) + h^2$ is negative ($d > D$ is automatically satisfied). If this is the case, then $\vartheta_A$ has simple roots at the CM points in $\mathcal{Q}_{|D|,d,h}$.
	
	As an example, we again consider $A = [1,1,1]$ with discriminant $D = -3$. There are precisely four tuples $(\alpha,\beta,\gamma) \in \Z^3$ satisfying the above conditions, $(\alpha,\beta,\gamma)= (2,-1,2)$, $(2,-3,6)$, $(2, 1,2)$, $(6,-3,2)$. For $(\alpha,\beta,\gamma)=(2,-1,2)$ we compute $h = 3$ and $d = -3$, and for the other three tuples we obtain $h = 5$ and $d = 13$. Hence, $\vartheta_{[1,1,1]}$ has simple roots precisely at the level $3$ CM points in $\mathcal{Q}_{3,-3,3}$.
\end{remark}

This work is organized as follows. We start with a short section on the necessary preliminaries about vector-valued modular forms for the Weil representation. In Section~\ref{section O2m Borcherds product} we construct, for any $m \in \N$, an automorphic product $\Psi_m$ on $O(2,m+2)$, which we identify in Section~\ref{section theta constants} (for $m = 1,2,3$) with suitable theta constants whose restrictions give binary/quaternary/octanary theta functions on $\H$. This yields a generalization of Theorem~\ref{theorem product expansion intro} to quaternary and octonary theta functions. In Section~\ref{section theta functions Borcherds products} we derive the product expansion in Theorem~\ref{theorem Borcherds products intro} and its generalization to quaternary and octonary theta functions. Finally, in Section~\ref{section counterexamples} we discuss some examples of theta functions which are not Borcherds products.

{\bf Acknowledgement.}
The first author was supported by SNF projects 200021\_185014 and PZ00P2\_202210. The second author was affiliated with TU Darmstadt during part of the work on this project and received funding from the LOEWE research unit USAG (Uniformized Structures in Arithmetic and Geometry).

\section{Preliminaries - Vector-valued modular forms for the Weil representation}

Let $L$ be an even lattice of signature $(b^+,b^-)$ with quadratic form $Q: L \to \Z$ and associated bilinear form $(x,y) = Q(x+y)-Q(x)-Q(y)$. For $N \in \Z$ we let $L(N)$ denote the rescaled lattice $(L,NQ)$.

Let $L'$ be the dual lattice of $L$. The discriminant form $L'/L$ is a finite quadratic module with the quadratic form $Q(x+L) = Q(x) \pmod \Z$. Let $\C[L'/L]$ be the group algebra of $L'/L$, which is generated by the formal basis vector $\e_{\gamma}$ for $\gamma \in L'/L$. It is equipped with a natural inner product $\langle \cdot ,\cdot\rangle$ which is antilinear in the second variable.

Let $\Mp_2(\Z)$ be the integral metaplectic group, consisting of pairs $\left(M,\phi \right)$ with $M = \left(\begin{smallmatrix}a & b \\ c & d \end{smallmatrix}\right) \in \SL_2(\Z)$ and $\phi: \H \to \C$ holomorphic with $\phi(\tau)^2 = c\tau + d$. The Weil representation $\rho_L$ associated with $L$ is a unitary representation of $\Mp_2(\Z)$ on $\C[L'/L]$ which is defined on the generators $T = \left(\left(\begin{smallmatrix}1 & 1 \\ 0 & 1 \end{smallmatrix}\right),1\right)$ and $S = \left(\left(\begin{smallmatrix}0 & -1 \\ 1 & 0 \end{smallmatrix}\right),\sqrt{\tau} \right)$ by
\[
\rho_L(T)\e_\gamma = e(Q(\gamma))\e_\gamma, \qquad \rho_L(S)\e_\gamma = \frac{e((b^- - b^+)/8)}{\sqrt{|L'/L|}}\sum_{\beta \in L'/L}e(-( \beta,\gamma )) \e_\beta,
\]
where we put $e(x) = e^{2\pi i x}$ for $x \in \C$.
Note that $\rho_L$ only depends on the discriminant form $L'/L$ rather than on the lattice $L$ itself.

A function $F: \H \to \C[L'/L]$ is called a weakly holomorphic modular form of weight $k \in \frac{1}{2}\Z$ for $\rho_L$ if it is holomorphic on $\H$, if it is invariant under the slash operator
\[
F|_{k,L}(M,\phi) = \phi^{-2k}\rho_L(M,\phi)^{-1}F(M\tau)
\]
for all $(M,\phi) \in \Mp_2(\Z)$, and if it is meromorphic at $\infty$. The latter condition means that $F$ has a Fourier expansion of the form
\[
F(\tau) = \sum_{\gamma \in L'/L}\sum_{\substack{n \in \Z + Q(\gamma) \\ n \gg -\infty}}c_F(\gamma,n)q^n \e_\gamma
\]
with coefficients $c_F(\gamma,n) \in \C$. We denote the space of weakly holomorphic modular forms of weight $k$ for $\rho_L$ by $M_{k,L}^!$.

\section{A Borcherds product on $O(2,m+2)$}
\label{section O2m Borcherds product}

In this section we construct, for any $m \in \N$, a Borcherds product $\Psi_m(Z)$ on $O(2,m+2)$ which corresponds to the weight $-m/2$ weakly holomorphic modular form $2^{m-1}/\theta^m$. For $m=1,2,3$ we will see later that $\Psi_m(Z)$ can be identified with a theta constant whose restrictions yield binary, quaternary, or octonary theta functions on $\mathcal{H}$.

For $m \in \N$ we consider the lattice
\[ 
K = U(4)\oplus m A_1(-1).
\]
Here $U(4) \cong (\Z^2, Q(\alpha,\gamma) = 4\alpha\gamma)$ denotes a rescaled hyperbolic plane, and $m A_{1}(-1)$ denotes $m$ copies of the unary lattice $A_1(-1) \cong (\Z,Q(\beta) = -\beta^2)$. Note that $K$ has signature $(1,m+1)$ and level $4$.

We will realize $K$ as $\Z \times \Z^m \times \Z$ with the quadratic form
\[
Q(\alpha,\beta_1,\dots,\beta_m,\gamma) = 4\alpha \gamma -\beta_1^2 - \ldots - \beta_m^2.
\]
Then its dual lattice is given by
\begin{align}\label{Kprime}
K' = \left\{(\alpha/4,\beta_1/2,\dots,\beta_m/2,\gamma/4)\, : \, \alpha,\gamma, \beta_1,\dots,\beta_m \in \Z\right\}.
\end{align}
In particular, we have $K'/K \cong (\Z/4\Z)^2 \times (\Z/2\Z)^m$.

We also define the lattice
\[
L = U \oplus K,
\]
where $U$ denotes an integral hyperbolic plane. Since $U$ is a unimodular even lattice of signature $(1,1)$, we see that $L$ has signature $(2,m+2)$, level $4$, and $L'/L \cong K'/K$. Therefore, we will not mention the variables of $U$ when we work in $L'/L$. Moreover, note that vector-valued modular forms for the Weil representation $\rho_L$ are the same as modular forms for $\rho_K$.

For the rest of this section, fix the element
\[
\mu = (1/2,1/2,\dots,1/2,1/2) \in K'/K \cong L'/L;
\]
that is, $\mu$ represents the class in $K'/K$ of vectors $(\alpha/4,\beta_1/2,\dots,\beta_m/2,\gamma/4) \in K'$ with $\alpha, \gamma \equiv 2 \pmod 4$ and $\beta_1,\dots,\beta_m$ all odd.

We now will construct a vector-valued modular form for the Weil representation $\rho_L$ by averaging the scalar-valued modular form $2^{m-1}/\theta^m$.

\begin{proposition}\label{proposition Borcherds input}
	Consider the following weakly holomorphic modular forms of weight $-1/2$:
	\begin{align*}
f &= \frac{1}{\theta} = 1 -2q + 4q^{2} -8q^{3} + 14q^{4} - 24q^{5} + 40q^{6} - 64q^{7} + \dots  \\
g &= \frac{1}{\sqrt{2i}}f|_{-1/2}S = f(\tau/4) = 1 -2q^{1/4} + 4q^{1/2} -8q^{3/4} + 14q - 24q^{5/4} + \dots \\ 
h &= -2if|_{-1/2}ST^2 S = \frac{\eta(2\tau)}{\eta(4\tau)^{2}} = q^{-1/4}\big(1-q^{2} + q^{4} - 2q^{6} + 3q^{8} - 4q^{10} +  \dots \big) 
\end{align*}
	Write $g^m = g_{0}^m + g_{1/4}^m+g_{1/2}^m+g_{3/4}^m$ where $g_{j}^m$ is the subseries\footnote{Note that by $g_j^m$ we really mean $(g^m)_j$, which is in general not the same as $(g_j)^m$ the $m$-th power of $g_j$ in the splitting $g = g_0+\dots+g_{3/4}$.} of $g^m$ consisting of terms $q^{n}$ with $n \equiv j \pmod 1$. We consider the $\C[L'/L]$-valued function
\begin{align*}
F_m &= h^m\e_{\mu} + \sum_{\lambda \in 2(L'/L)}\sgn(\lambda/2)2^{m-1}f^m\e_{\lambda} - \sum_{\lambda \in L'/L}  \sgn(\lambda) 2^m g_{Q(\lambda)}^m\e_{\lambda},
\end{align*}
where we wrote $2(L'/L) = \{2\nu: \nu \in L'/L\}$ and $\sgn(\lambda) = (-1)^{(\alpha+1)(\gamma+1)}$ for $\lambda \in K'$ as in \eqref{Kprime}. Then $F_m$ is a weakly holomorphic modular form of weight $-m/2$ for the Weil representation $\rho_L$. Its principal part is supported on $\e_\mu$, and is given by the principal part of $h^m$. Moreover, $F_m$ has constant term $c_{F_m}(0,0) = 2^{m-1}$ and integral Fourier coefficients.
\end{proposition}

\begin{proof}
	Throughout the proof we will identify $L'/L$ with $K'/K$, and write elements $\lambda \in L'/L$ in the form $\lambda = (\alpha/4,\beta_1/2,\dots,\beta_m/2,\gamma/4)$ with $\alpha,\gamma \in \Z/4\Z$ and $\beta_1,\dots,\beta_m \in \Z/2\Z$. Note that there are precisely four elements $\lambda \in 2(L'/L)$, namely those with $\beta_1 ,\dots,\beta_m$ all even, and $\alpha,\gamma \in \Z/4\Z$ both even. For each such $\lambda \in 2(L'/L)$ we consider the weakly holomorphic vector-valued modular form
	\[
	G_{\lambda} = \sum_{\tilde{M} \in \tilde{\Gamma}_{0}(4)\setminus \Mp_{2}(\Z)}(2^{m-1}/\theta^m)|_{-m/2}\tilde{M} \ \rho_{L}(\tilde{M})^{-1}\e_{\lambda} \in M_{-m/2,L}^!,
	\]
	of weight $-m/2$ for $\rho_L$. This scalar-to-vector-valued lift has been studied in detail in \cite{scheithauer}. Note that, for the lift to be well-defined, it is important that $2\lambda = 0$ in $L'/L$ and $Q(\lambda) = 0 \pmod \Z$.
		
	We consider the modular form
	\[
	F_m = \sum_{\lambda \in 2(L'/L)}\sgn(\lambda/2)G_\lambda \in M_{-m/2,L}^!,
	\]
	and show that is has an expansion as stated in the proposition. Note that $\sgn(\lambda/2)$ equals $-1$ if $\lambda = 0$ and $1$ otherwise. Fix $\lambda \in 2(L'/L)$. We have $[\SL_{2}(\Z):\Gamma_{0}(4)] = 6$, and representatives are given by
	\begin{align*}
	I = \begin{pmatrix}1 & 0 \\ 0 & 1 \end{pmatrix}, \quad	ST^{2}S = \begin{pmatrix} -1 & 0\\ 2 & -1\end{pmatrix}, \quad
	ST^j = \begin{pmatrix} 0 & -1 \\ 1 & j \end{pmatrix} \ (0 \leq j \leq 3).
	\end{align*}
	Hence $G_\lambda$ is given by
	\[
	G_\lambda = 2^{m-1}f^m \e_\lambda + 2^{m-1}(-2i)^{-m}h^m \rho_L(ST^2S)^{-1}\e_\lambda + \sum_{j=0}^3 2^{m-1}(\sqrt{2i})^m g^m(\tau+j) \rho_L(ST^j)^{-1}\e_\lambda.
	\]
	Now we compute the action of the Weil representation. Recall that $|L'/L| = |K'/K| = 2^{m+4}$ and $L$ has signature $(2,m+2)$. Hence, by the definition of the Weil representation $\rho_L$, we have for any $j \in \Z$,
	\[
	\rho_L(ST^j)^{-1}\e_\lambda = \frac{\sqrt{i}^{-m}}{\sqrt{2^{m+4}}}\sum_{\delta \in L'/L}e((\delta,\lambda)-jQ(\delta))\e_\delta.
	\]
	Similarly, we compute
	\[
	\rho_L(ST^2 S)^{-1}\e_\lambda = \frac{i^{-m}}{2^{m+4}}\sum_{\nu \in L'/L}\sum_{\delta \in L'/L}e((\delta,\lambda+\nu)-2Q(\delta))\e_\nu.
 	\]
	For fixed $\nu \in L'/L$ the inner sum can be computed if we write $\delta = (\alpha/4,\beta_1/2,\dots,\beta_m/2,\gamma/4)$ with $\alpha,\gamma \in \Z/4\Z$ and $\beta_1,\dots,\beta_m \in \Z/2\Z$ (and similarly for $\lambda + \nu$). The result is 
	\[
	\sum_{\delta \in L'/L}e((\delta,\lambda+\nu)-2Q(\delta)) = \begin{cases}-2^{m+3}, & \text{if $\nu = \lambda+\mu$ mod $L$,} \\
	2^{m+3}, & \text{if $\nu \in 2(L'/L) + \mu$ but $\nu \neq \lambda + \mu$ mod $L$}, \\
	0, &  \text{otherwise.}
	\end{cases}
	\]
	Combining the above formulas, it is now straightforward to check that $F_m$ has the expansion in the proposition. The remaining claims about the principal part and the Fourier coefficients of $F_m$ follow immediately.
\end{proof}

%

We let $C \subset K \otimes \R$ be the positive cone consisting of vectors of positive norm with $\alpha > 0$, and we let $\H_{m+1} = (K \otimes \R) \oplus iC$ be the upper half-space model of the hermitian symmetric domain for $O(2,m+2)$. We will write $Z \in \mathcal{H}_{m+1}$ as $Z = (w,-z/2,\tau)$ with $w,\tau \in \H$ and $z = (z_1,\dots,z_m) \in \C^m$.

\begin{theorem}\label{theorem O2m Borcherds product}
	For each $m \in \N$ there exists a (meromorphic) automorphic product $\Psi_m(\tau,z,w)$ of weight $2^{m-2}$ on $\H_{m+1}$ with expansion
	\[
	\Psi_m(\tau,z,w) = \!\!\!\!\!\prod_{\substack{\alpha,\beta_1,\dots,\beta_m,\gamma \in \Z \\ \alpha +\gamma > 0}}\!\!\left(\frac{1+(-1)^{(\alpha+1)(\gamma+1)}q^{\alpha/2} r_1^{\beta_1} \cdots r_m^{\beta_m} s^{\gamma/2}}{1-(-1)^{(\alpha+1)(\gamma+1)}q^{\alpha/2} r_1^{\beta_1} \cdots r_m^{\beta_m} s^{\gamma/2}}\right)^{c_{2^{m-1}/\theta^m}(\alpha \gamma -\beta_1^2 - \ldots - \beta_m^2)}
	\]
	with $q = e^{2\pi i \tau}, r_j = e^{2\pi i z_j}, s = e^{2\pi i w}$. For $m \leq 8$, it is holomorphic.
\end{theorem}

%

\begin{proof}
	 We apply Borcherds' Theorem 13.3 \cite{borcherds} to the weakly holomorphic modular form $F_m$ given in Proposition~\ref{proposition Borcherds input}. Then the weight and the singularities of $\Psi_{m}$ are clear from the explicit Fourier expansion of $F_m$.
	
	For $Z = X+iY \in \mathcal{H}_{m+1}$ with $Q(Y) \gg 0$ large enough we have the product expansion
		\begin{align}\label{eq product expansion}
		\Psi_{m}(Z) = \prod_{\substack{\lambda \in K' \\ (\lambda,C) > 0}}\big(1-e((\lambda,Z))\big)^{c_{F_m}(\lambda,Q(\lambda))}.
		\end{align}
		This also follows from Borcherds' Theorem 13.3 once we have shown that there is only a single Weyl chamber, namely the whole positive cone $C$, and the Weyl vector vanishes. To this end, it suffices to show that the set
	\[
	K_{\mu,n} = \{\lambda \in K + \mu: Q(\lambda) = n\}
	\]
	is empty whenever $(\eta(2\tau)/\eta(4\tau)^2)^m[n] \neq 0$. From the Fourier expansion of $\eta(2\tau)/\eta(4\tau)^2$ given in Proposition~\ref{proposition Borcherds input} we see that we only need to consider $n$ of the form $n = -m/4+2\ell$ for $\ell \in \Z$. For such $n$, suppose that $K_{\mu,n} \neq \emptyset$, and let $\lambda \in K_{\mu,n}$. Since $\mu = (1/2,1/2,\dots,1/2,1/2)$, we can write $\lambda = (\alpha/2,\beta_1/2,\dots,\beta_m/2,\gamma/2)$ with $\alpha,\gamma,\beta_1,\dots,\beta_m$ all odd. We find
	\[
	-m/4+2\ell = n = Q(\lambda) = \alpha \gamma - \beta_1^2/4-\ldots-\beta_m^2/4.
	\]
	Multiplying both sides by $4$, reducing modulo $8$, and using that $4\alpha \gamma \equiv 4 \pmod 8, \beta_j^2 \equiv 1 \pmod 8$, we obtain a contradiction. This yields \eqref{eq product expansion}.

	Finally, we will simplify the product expansion in \eqref{eq product expansion}. Using the Fourier expansion of $F_m$ from Proposition~\ref{proposition Borcherds input}, we find that the exponents in \eqref{eq product expansion} are given by
	\begin{align*}
	c_{F_m}(\lambda,Q(\lambda)) &= \begin{cases}
	(-\sgn(\lambda)2^m g_{Q(\lambda)}^m+\sgn(\lambda/2)2^{m-1}f^m)[Q(\lambda)], &\lambda \in 2K', \\
	-\sgn(\lambda)2^m g_{Q(\lambda)}^m[Q(\lambda)], & \lambda \notin 2K'.
	\end{cases}
	\end{align*}
	Note that we have $g_{Q(\lambda)}^m[Q(\lambda)] = g^m[Q(\lambda)]=f^m[4Q(\lambda)]$. Hence, after a short computation the product in \eqref{eq product expansion} can be written as
	\begin{align*}
	\prod_{\substack{\lambda \in K' \\ (\lambda,C) > 0 }}\left(\frac{1-e(2(\lambda,Z))}{(1-e((\lambda,Z)))^2}\right)^{\sgn(\lambda)2^{m-1} f^m[4Q(\lambda)]} = \prod_{\substack{\lambda \in K' \\ (\lambda,C) > 0 }}\left(\frac{1+e((\lambda,Z))}{1-e((\lambda,Z))}\right)^{\sgn(\lambda)2^{m-1} f^m[4Q(\lambda)]}.
	\end{align*}
	By \cite[Lemma~3.2]{bruinierhabil}, the condition $(\lambda,C) > 0$ may be replaced by $(\lambda,v) > 0$ for any $v \in C$. For example, we may choose $v = (1,0,\dots,0,1) \in C$, so the condition $(\lambda,C) > 0$ becomes $\alpha + \gamma > 0$. Plugging in $(\lambda,Z) = \alpha \tau + \gamma w+\beta_1 z_1/2+\ldots+\beta_m z_m/2$ for $Z = (w,-z/2,\tau) \in \mathcal{H}_{m+1}$, $4Q(\lambda) = \alpha \gamma-\beta_1^2-\ldots-\beta_m^2$, $\sgn(\lambda) = (-1)^{(\alpha+1)(\gamma+1)}$, and $f = 1/\theta$, we obtain the product expansion given in Theorem~\ref{theorem O2m Borcherds product}.
\end{proof}

\section{Theta constants on $O(2,m+2)$ as Borcherds products}
\label{section theta constants}

We now show that, for $m =  1,2,3$, the Borcherds product $\Psi_m(Z)$ from Theorem~\ref{theorem O2m Borcherds product} is given by a theta constant on the degree $2$ Siegel, hermitian, or quaternionic upper half-space, respectively.

We let $m=1,2,$ or $3$, and we let $\mathcal{O}_m$ be equal to $\Z,\Z[i]$, or $\Z[i,j,k] \subset \mathbb{H}$, respectively. As before, we let $\mathcal{H}_{m+1}$ be the upper half-space model of the hermitian symmetric domain for $O(2,m+2)$, and we denote its variables by $Z = (w,-z,\tau)$ with $\tau,w \in \mathcal{H}$ and $z = (z_1,\dots,z_m) \in \C^m$. We define the theta constant
\[
\Theta_m(Z) = \sum_{\lambda,\mu \in \mathcal{O}_m}e(|\lambda|^2\tau/2 + \tr(\lambda z\overline{\mu})/2 + |\mu|^2 w/2),
\]
which can be viewed as a degree $2$ Siegel theta constant of weight $1/2$ if $m = 1$ (see \cite[Satz I.3.2]{freitagsiegel}), as a degree $2$ hermitian theta constant of weight $1$ if $m = 2$ (see \cite[Satz III.1]{freitaghermitian}), and as a degree $2$ quaternionic theta constant of weight $2$ if $m = 3$ (see \cite[Proposition 10.4]{freitaghermann}; see also the book \cite{kriegquaternionic} for a comprehensive treatment of theta functions on quaternionic half-spaces). The theta constant is related to the Borcherds product $\Psi_m$ as follows.

\begin{proposition}\label{proposition theta constants} We have
	\begin{align*}
		\Theta_1(Z) = \Psi_1(Z), \qquad
		\Theta_2(Z) = \Psi_2(Z),\qquad
		\Theta_3(Z)|_{z_4 = 0} = \Psi_3(Z).
	\end{align*}
\end{proposition}

\begin{proof}
For $m = 1$ and $m = 2$ this is well-known (compare \cite{hermann, hermannhermitian}, for example), so we focus on the case $m = 3$ here. It is also possible to prove the identities for $m=1$ and $m=2$ using a slight variation of this argument. 

Since $\Psi_3$ has an $\mathrm{O}^+(L)$-invariant divisor that does not contain the mirror of any reflections, it defines a modular form for a finite-order character $\chi$ on $\mathrm{O}^+(L)$ which is trivial on all reflections. Let $\Gamma := \mathrm{ker}(\chi)$. The values of $\Psi_3$ at the one-dimensional cusps of $\Gamma$ are Borcherds products of weight two on $U \oplus U(4)$ and therefore modular forms on a subgroup of $\mathrm{SL}_2(\mathbb{Z}) \times \mathrm{SL}_2(\mathbb{Z})$, and to prove that they coincide with the values of $\Theta_3(Z)|_{z_4=0}$ it is enough to check that a few Fourier coefficients match. Therefore $$F(Z) := \Theta_3(Z)|_{z_4=0} - \Psi_3(Z)$$ is a cusp form of weight two under $\Gamma$. 

If $m$ were $1$ or $2$ then this would immediately imply the identity, as a cusp form of singular weight is identically zero. In our case we must argue more carefully that $S_2(\Gamma) = 0$.

Consider the Fourier--Jacobi expansion, $$F(\tau, z, w) = \sum_{n=N}^{\infty} f_n(\tau, z) e^{2\pi i n w},$$ and suppose $N$ is minimal with $f_N \ne 0$. Then $f_N$ is a Jacobi cusp form of weight two, level $\Gamma_0(4)$ and lattice index $3A_1(N)$. In view of the expansion of $F$ about the cusps, and the fact that $\sqrt{\Delta(2\tau)} = q - 12q^3 + 54q^5 \pm ...$ generates all cusp forms of level $\Gamma_0(4)$ as a principal ideal, we find that $f_N / \Delta(2\tau)^{N/2}$ is a weak Jacobi form of weight $2-6N$ and level $\Gamma_0(4)$, and therefore lies in the algebra $$f_N(\tau, z_1,z_2,z_3) \in M_*(\Gamma_0(4))[\phi_{0,1}(\tau, z_i), \phi_{-2,1}(\tau, z_i), \; i=1,2,3],$$ where $\phi_{0,1}$ and $\phi_{-2,1}$ are the basic weak Jacobi forms defined by Eichler--Zagier \cite[Theorem~9.3]{eichlerzagier}. Further, since $F$ is symmetric under the reflections that swap any variables $z_i, z_j$, it follows that $f_N$ is a symmetric polynomial with respect to $z_i$; this determines $f_N$ as a linear combination $$\frac{f_N}{\Delta(2\tau)^{N/2}} = \Psi_{-6}^N \cdot \Big( A \cdot \Psi_{-4} + B \cdot \Psi_{-6} \Big),$$ where $$\Psi_{-6}(\tau, z_1, z_2, z_3) =  \phi_{-2,1}(\tau, z_1)\phi_{-2, 1}(\tau, z_2) \phi_{-2, 1}(\tau, z_3)$$ and \begin{align*} \Psi_{-4}(\tau, z_1, z_2, z_3) &= \phi_{-2,1}(\tau, z_1)\phi_{-2, 1}(\tau, z_2) \phi_{0, 1}(\tau, z_3) + \phi_{-2,1}(\tau, z_1)\phi_{0, 1}(\tau, z_2) \phi_{-2, 1}(\tau, z_3) \\ & \quad + \phi_{0,1}(\tau, z_1)\phi_{-2, 1}(\tau, z_2) \phi_{-2, 1}(\tau, z_3),\end{align*} where $A \in\mathbb{C}$, and where $B \in M_2(\Gamma_0(4))$. However, no such linear combination yields a cusp form $f_N$, which is a contradiction.
\end{proof}

Let us now discuss a generalization of Theorem~\ref{theorem product expansion intro} to theta functions associated with integral binary hermitian forms over $\Z[i] \subset \C$ and $\Z[i,j,k] \subset \mathbb{H}$. For $a,c \in \N$ and $b = b_1+ib_2 \in \Z[i]$ with $D = |b|^2-4ac < 0$ we consider the positive definite, binary hermitian form
\begin{align}\label{binary hermitian form}
A(x,y) = a|x|^2 + \tfrac{1}{2}\tr(xb\overline{y})+ c|y|^2 \qquad (x,y \in \C).
\end{align}
By splitting $x = x_1 + ix_2$ and $y = y_1 + iy_2$ into their real and imaginary parts, we obtain the positive definite, integral quaternary quadratic form
\[
A(x_1,x_2,y_1,y_2) = a(x_1^2+x_2^2) + b_1(x_1 y_1 + x_2 y_2) + b_2(x_1 y_2 - x_2 y_1) + c(y_1^2+y_2^2)
\]
of discriminant $D^2$ and level dividing $|D|$. Hence, the corresponding quaternary theta function 
\[
\vartheta_{A}(\tau) = \sum_{x,y \in \Z[i]}q^{A(x,y)} = \sum_{x_1,x_2,y_1,y_2 \in \Z}q^{A(x_1,x_2,y_1,y_2)}
\]
is a modular form of weight $2$ for $\Gamma_0(|D|)$ with trivial nebentypus.

Similarly, for $a,c \in \N$ and $b = b_1+ib_2+b_3 j + b_4 k \in \Z[i,j,k]$ with $D = |b|^2-4ac < 0$, equation \eqref{binary hermitian form} defines a positive definite, binary hermitian form $A(x,y)$ over the quaternions $\mathbb{H}$. Splitting $x$ and $y$ into their four components, we obtain a positive definite, octonary quadratic form and a corresponding weight $4$ theta function for $\Gamma_0(|D|)$.

By evaluating the theta constant $\Theta_2(Z) = \Psi_2(Z)$ at $Z = (2a\tau,b_1 \tau,b_2 \tau,2c\tau)$ (and similarly for $\Theta_3(Z)|_{z_4 = 0}$) we obtain the following analog of Theorem~\ref{theorem product expansion intro} for the above weight $2$ and $4$ theta functions.

\begin{theorem}\label{theorem product expansion hermitian intro}
	\begin{enumerate}
		\item Let $A = [a,b,c]$ be a positive definite, integral binary hermitian form over $\Z[i]$ as above. For $\Im(\tau)$ large enough the associated weight $2$ quaternary theta function $\vartheta_A$ has a product expansion as in \eqref{product expansion}, but with the exponents
	\[
	d_A(n) = \!\!\!\!\sum_{\substack{\alpha, \beta_1,\beta_2, \gamma \in \Z \\ \alpha a + \beta_1 b_1 + \beta_2 b_2 + \gamma c  = n}}\!\!\!\!(-1)^{(\alpha+1)(\gamma+1)}c_{2/\theta^2}(\alpha\gamma-\beta_1^2-\beta_2^2).
	\]
	\item Let $A = [a,b,c]$  be a positive definite, integral binary hermitian form over $\Z[i,j,k]$ as above. Write $b = b_1 + b_2 i + b_3 j + b_4 k$ and assume that $b_4 = 0$. For $\Im(\tau)$ large enough the associated weight $4$ octonary theta function $\vartheta_A$ has a product expansion as in \eqref{product expansion}, but with the exponents
	\[
	d_A(n) = \!\!\!\!\sum_{\substack{\alpha, \beta_1,\beta_2,\beta_3, \gamma \in \Z \\ \alpha a + \beta_1 b_1 + \beta_2 b_2+\beta_3 b_3 + \gamma c  = n}}\!\!\!\!(-1)^{(\alpha+1)(\gamma+1)}c_{4/\theta^3}(\alpha\gamma-\beta_1^2-\beta_2^2-\beta_3^2).
	\]
	\end{enumerate}
\end{theorem}



Let us look at two examples.

\begin{example}\label{example root lattices}
	\begin{enumerate}
		\item The binary hermitian form $A = [a,b,c] = [1,1+i,1]$ over $\Z[i]$ has discriminant $-2$, and the corresponding quaternary quadratic form is the $D_4$ root lattice, which has discriminant $4$ and level $2$. Therefore, the corresponding theta function is a modular form of weight $2$ for $\Gamma_0(2)$, and is given by the Eisenstein series $\vartheta_{D_4}(\tau) = 2 E_2(2\tau) - E_2(\tau)$. From Theorem~\ref{theorem product expansion hermitian intro} we obtain the product expansion
	\[
	\vartheta_{D_4}(\tau) = \prod_{n = 1}^\infty\left(\frac{1+q^n}{1-q^n} \right)^{d(n)}, \quad d(n) = \!\!\!\!\sum_{\substack{\alpha, \beta_1,\beta_2 \gamma \in \Z \\ \alpha + \beta_1 + \beta_2 + \gamma  = n}}\!\!\!\!(-1)^{(\alpha+1)(\gamma+1)}c_{2/\theta^2}(\alpha\gamma-\beta_1^2-\beta_2^2).
	\]

		\item The binary hermitian form $A = [a,b,c] = [1,1+i+j,1]$ over $\Z[i,j,k]$ has discriminant $-1$, and the corresponding octonary quadratic form is the $E_8$ root lattice. The associated octonary theta function $\vartheta_{A} = \vartheta_{E_8}$ is a holomorphic modular form of weight $4$ for $\SL_2(\Z)$. Hence, it is equal to the normalized Eisenstein series $E_4$. From Theorem~\ref{theorem product expansion hermitian intro} we obtain the product expansion
	\[
	\vartheta_{E_8}(\tau) = \prod_{n = 1}^\infty\left(\frac{1+q^n}{1-q^n} \right)^{d(n)}, \quad d(n) = \!\!\!\!\sum_{\substack{\alpha, \beta_1,\beta_2,\beta_3, \gamma \in \Z \\ \alpha + \beta_1 + \beta_2+\beta_3 + \gamma  = n}}\!\!\!\!(-1)^{(\alpha+1)(\gamma+1)}c_{4/\theta^3}(\alpha\gamma-\beta_1^2-\beta_2^2-\beta_3^2).
	\]
	Although it is well-known that $\vartheta_{E_8} = E_4$ is a Borcherds product (see \cite{borcherds95}), this particular product expansion seems to be new.
	\end{enumerate}
\end{example}

\section{Binary theta functions as Borcherds products - The proof of Theorem~\ref{theorem Borcherds products intro}}
\label{section theta functions Borcherds products}


We let $m=1,2,$ or $3$, and we let $A = [a,b,c]$ be a positive definite, binary quadratic or hermitian form over $\Z$, $\Z[i]$ or $\Z[i,j,k]$ (with $b_4 = 0$ in the last case) as in the last section. In particular, we have $a,c \in \N$ and either $b = b_1 \in \Z$, $b = b_1 + ib_2 \in \Z[i]$, or $b = b_1 + b_2 i + b_3j \in \Z[i,j,k]$.

We consider the vector
\[
A = (a,b_1,\dots,b_m,c) \in K, \qquad Q(A) = 4ac - |b|^2 = |D|.
\]
We define the positive definite rank $1$ and negative definite rank $m+1$ sublattices
\[
P = \Z A, \quad N = A^\perp \cap K,
\]
where $A^\perp$ denotes the orthogonal complement of $A$ in $K \otimes \R$. Clearly, we have $P \cong A_1(|D|)$.

\begin{remark}
	For $m = 1$ the lattice $N(-1) = (N,-Q(x))$ is positive definite of rank $2$, and is related to the binary quadratic form $A$ as follows. If $D$ is odd, then $N(-1)$ is isomorphic to the binary quadratic form $A^2(2x,2y)$, where $A^2$ is the square of $A$ with respect to Gauss composition. If $D$ is even, then $N(-1)$ is isomorphic to $A^2(x,2y)$. This can be checked by a direct computation after choosing a basis of $N(-1)$.
\end{remark}

Let $F_m$ be the weakly holomorphic modular form of weight $-m/2$ for $\rho_L$ constructed in Proposition~\ref{proposition Borcherds input}. Since $L'/L \cong K'/K$, we can view $F_m$ as a modular form for $\rho_K$. Note that $P \oplus N$ is a sublattice of $K$ of finite index, and we have $K' \subset (P \oplus N)'$. Hence, we may view $F_m$ as a modular form $F_{m,P \oplus N}$ of weight $-m/2$  for $
\rho_{P \oplus N}$ by setting
\[
F_{m,P \oplus N, \beta} = \begin{cases}
F_{m,\beta}, & \text{if } \beta \in K'/(P\oplus N), \\
0, & \text{otherwise}. 
\end{cases}
\]
We refer the reader to \cite[Section~2.2]{ma} for more details on this construction. We have the following Borcherds product expansion for the theta function $\vartheta_A$.

\begin{theorem}\label{theorem Borcherds products}
	The binary/quaternary/octonary theta function $\vartheta_A(\tau)$ is the Borcherds product associated to the weakly holomorphic modular form
	\begin{align}\label{definition FA}
	F_A = \sum_{\beta \in P'/P}\left(\sum_{\alpha \in N'/N}F_{m,P \oplus N, \alpha+\beta}
	\Theta_{N(-1),\alpha}\right)\e_\beta\in M_{1/2,P}^!,
	\end{align}
	where 
	\[
	\Theta_{N(-1)} = \sum_{\alpha \in N'/N}\Theta_{N(-1),\alpha}\e_\alpha = \sum_{\alpha \in N'/N}\sum_{x \in N + \alpha}q^{-Q(x)}\e_\alpha \in M_{(1+m)/2,N(-1)}
	\]
	is the holomorphic vector-valued binary theta function associated to $N(-1)$. In particular, for $\Im(\tau)$ large enough, the theta function $\vartheta_A(\tau)$ has the Borcherds product expansion
	\[
	\vartheta_A(\tau) = \prod_{n \geq 1}(1-q^n)^{c_{A}(n,n^2/4|D|)},
	\]
	where we identified $P'/P \cong \Z/2|D|\Z$, and we wrote $F_A = \sum_{h (2|D|)}\sum_{\ell \in \Z + h^2/4|D|}c_A(h,\ell)q^\ell \e_h$ for the Fourier expansion of $F_A$. Explicitly, we have
	\[
	c_A(n,n^2/4|D|) = \sum_{\substack{\lambda \in K' \\ (\lambda,A) = n}}c_{F_m}(\lambda,Q(\lambda)).
	\]
\end{theorem}

\begin{remark} The theorem implies that the above binary/quaternary/octonary theta functions $\vartheta_A(\tau)$ have roots only at CM points in $\mathcal{H}$ that can be described explicitly (see also \cite{hermann,hermannhermitian}).
\end{remark}

In order to prove the above theorem, we first compute the Fourier expansion of the weight $1/2$ Borcherds input $F_A$ for $\vartheta_A$.

\begin{lemma}\label{lemma fourier expansion FA}
	For $h \in \Z/2|D|\Z$ and $\ell \in \Z + h^2/4|D|$, the $(h,\ell)$-th Fourier coefficient of $F_A$ is given by
	\[
	c_A(h,\ell) = \sum_{\substack{\lambda \in K' \\ (\lambda,A) = h}}c_{F_m}(\lambda,\ell+Q(\lambda)-h^2/4|D|),
	\]
	where we have identified $P'/P \cong \Z/2|D|\Z$, and where $F_m$ is the vector-valued modular form constructed in Proposition~\ref{proposition Borcherds input}.
\end{lemma}


\begin{proof}
	The $\ell$-th coefficient of the $\beta$-component of $F_A$ is given by
	\begin{align*}
	\sum_{\substack{\alpha \in N'/N \\ \alpha + \beta \in K'}} \sum_{x \in N+\alpha} c_{F_m}(\alpha+\beta,\ell+Q(x)) &= \sum_{\substack{x \in N' \\ x+\beta \in K'}}c_{F_m}(x+\beta,\ell+Q(x)) \\
	&= \sum_{\substack{x \in K'-\beta \\ (x,A) = 0}}c_{F_m}(x+\beta,\ell+Q(x)).
	\end{align*}
	If we represent $\beta \in P'/P$ by $\frac{h}{2|D|}A$ with $h \in \Z/2|D|\Z$, and write $x = \lambda-\beta$ with $\lambda \in K'$, then we have $(\lambda,A) = (\beta,A) = h$ and $Q(x) = Q(\lambda)-h^2/4|D|$. This yields the stated formula.
\end{proof}

\begin{proof}[Proof of Theorem~\ref{theorem Borcherds products}]
	We know from Proposition~\ref{proposition theta constants} that $\vartheta_A(\tau)$ is the restriction to $Z = A\tau$ of the Borcherds product $\Psi_m(Z)$ from Theorem~\ref{theorem O2m Borcherds product}. By plugging in $Z = A\tau = (2a\tau,b_1\tau,\dots,b_m \tau,2c\tau)$ in the product expansion \eqref{eq product expansion} of $\Psi_m(Z)$, we obtain 
	\[
	\vartheta_A(\tau) = \Psi_m(A\tau) =  \prod_{\substack{\lambda \in K' \\ (\lambda,A) > 0}}\left(1-e((\lambda,A)\tau)\right)^{c_{F_m}(\lambda,Q(\lambda))} = \prod_{n = 1}^\infty (1-q^n)^{C(n)}
	\]
	where
	\[
	C(n) = \sum_{\substack{\lambda \in K' \\ (\lambda,A) = n}}c_{F_m}(\lambda,Q(\lambda)).
	\]
	Using the Fourier expansion of $F_A$ from Lemma~\ref{lemma fourier expansion FA}, we see that $C(n) = c_A(n,n^2/4|D|)$. This yields the product expansion of $\vartheta_A$ given in Theorem~\ref{theorem Borcherds products} and finishes the proof of the theorem.
\end{proof}

\begin{remark} In the quaternary and octonary cases, the theta function can be written as a Kohnen--Shimura lift by a similar construction (with a similar proof). In the case $m = 2$, we define a Weil invariant (or modular form of weight $0$) for $\rho_L$: $$G_2 = \sum_{\lambda \in (K'/K)^+} \mathfrak{e}_{\lambda} - \sum_{\lambda \in (K'/K)^-} \mathfrak{e}_{\lambda},$$ where $(K'/K)^{\pm}$ consists of norm-zero cosets $(\alpha/4, \beta/2, \gamma/4)$ for which $\alpha \equiv \pm 1$ mod $4$ or $\gamma \equiv \mp 1$ mod $4$ (or both). In the case $m = 3$, we define a weight $1/2$ modular form $G_3$ by setting \begin{align*} G_3 &= \theta_{00} \cdot \Big( 3 \mathfrak{e}_0 - \sum_{\substack{\lambda \in K'/K \\ Q(\lambda) = 0 \\ \lambda \ne 0}} \mathfrak{e}_{\lambda} \Big) + \theta_{10} \cdot \Big( -3 \mathfrak{e}_{\mu} -\sum_{\substack{\lambda \in K'/K \\ \mathrm{ord}(\lambda) = 4 \\ Q(\lambda) = -1/4}} \mathfrak{e}_{\lambda} + \sum_{\substack{\lambda \in K'/K \\ \mathrm{ord}(\lambda) = 2 \\ \lambda \ne \mu \\ Q(\lambda) = -1/4}} \mathfrak{e}_{\lambda} \Big), \end{align*} where as before $\mu = (1/2,...,1/2) \in K'/K$, and where $\theta_{00}$ and $\theta_{10}$ are the thetanulls $$\theta_{00}(\tau) = \sum_{n \in \mathbb{Z}} q^{n^2} = 1 + 2q + 2q^4 + \dots, \quad \theta_{10} = \sum_{n \in \frac{1}{2} + \mathbb{Z}} q^{n^2} = 2q^{1/4} + 2q^{9/4} + \dots$$ The (additive) theta lifts of $G_2$ and $G_3$ can be shown to equal the theta constants $\Theta_2$ and $\Theta_3(Z)|_{z_4=0}$ by verifying that their images under the Witt operators match, as argued in the proof of Proposition~\ref{proposition theta constants}. Then we find that $\vartheta_A(\tau)$ is the Kohnen--Shimura lift of the (holomorphic) vector-valued modular form $$G_A = \sum_{\beta \in P'/P} \Big( \sum_{\alpha \in N'/N} G_{m, P\oplus N, \alpha+\beta} \Theta_{N(-1), \alpha} \Big) \mathfrak{e}_{\beta} \in M_{2m - 5/2, P},$$ where $G_{m, P\oplus N, \beta}$ is defined analogously to $F_{m, P\oplus N, \beta}$. \\ \

This remark does not apply directly to binary theta functions, which have odd weight and are not Shimura lifts in any usual sense. However, it is possible to realize the squares $\vartheta_A^2$ of binary theta series uniformly as Shimura lifts. This implies an ``additive'' analogue of Theorem~\ref{theorem product expansion intro}, namely the Lambert series expansion $$\vartheta_A(\tau)^2 = 1 + 4 \sum_{n=1}^{\infty} d_A(n) \frac{q^n}{1 - q^n},$$ valid for all binary theta series, where $$d_A(n) = \sum_{\substack{\alpha, \beta, \gamma \in \mathbb{Z} \\ \alpha a + \beta b + \gamma c = n}} \chi\Big(\mathrm{gcd}(\alpha,\beta,\gamma)\Big) c_{\theta}(\alpha \gamma - \beta^2)$$ for $A(x,y) = ax^2 + bxy + cy^2$, where we have written $$\theta(\tau) = \sum_{n \ge 0} c_{\theta}(n)q^n = 1 + 2q + 2q^4 + 2q^9 + ...$$ and where $\chi$ is the nontrivial Dirichlet character modulo $4$. \\

The powers $\vartheta_A^4$ and $\vartheta_A^8$ can also be realized uniformly as Shimura lifts, as the fourth and eighth powers of the Siegel theta constant can also be realized as additive theta lifts.
\end{remark}

\section{Theta functions that are not Borcherds products}
\label{section counterexamples}

We close with some remarks on (counter-)examples of other theta functions that are (or are not) Borcherds products. 

\subsection{Even theta constants} For a positive definite, integral binary quadratic form $A$ and an even characteristic $(\mu,\nu)$ (that is, $\mu,\nu \in \{0,1\}^2$ with $\mu^t \nu$ even), the same arguments as in the proof of Theorem~\ref{theorem product expansion intro} can be used to show that each of the ten even theta constants
\[
\vartheta_{A,\mu,\nu}(\tau) = \sum_{m,n \in \Z}(-1)^{\nu_1 m + \nu_2 n}q^{A(m+\mu_1/2,n+\mu_2/2)}
\]
is a Borcherds product. Indeed, $\vartheta_{A,\mu,\nu}(\tau)$ is the restriction of a degree $2$ Siegel theta constant, which is known to be a Borcherds product (see e.g. \cite{gritsenkonikulin, hermann, lippolt, opitzschwagenscheidt}). A similar result holds for the ten even Hermitian theta characteristics, as well as the ten nonzero restrictions of even quaternionic theta characteristics to the hypersurface $z_4 = 0$.
	
\subsection{Theta functions with non-CM roots} A Borcherds product has roots and poles only at CM points, but holomorphic theta functions in three or more variables often do not. For example, the theta function $\vartheta_\Lambda$ associated with the Leech lattice $\Lambda$ has roots only at non-CM points, and hence cannot be a Borcherds product. Indeed, since $\Lambda$ is unimodular of rank $24$ and has no $2$-roots, its theta function is the unique modular form of weight $12$ for $\SL_2(\Z)$ with Fourier expansion beginning $1 + O(q^2)$. Using $\Delta(\tau) = q + O(q^2)$ and $\Delta(\tau) j(\tau) = 1 + 720q + O(q^2)$, we obtain 
\[
\vartheta_{\Lambda}(\tau) = \Delta(\tau) \cdot\big( j(\tau) - 720 \big).
\]
It follows that $j(\tau_0) = 720$ for every zero $\tau_0$ of $\vartheta_{\Lambda}$.

If $\tau_0$ were a quadratic irrational then by complex multiplication theory the order $\mathcal{O} = \mathbb{Z} + \mathbb{Z}\tau_0$ has class number one. The Stark--Heegner theorem implies that $\tau_0$ falls into one of $13$ $\SL_2(\Z)$-orbits at which the $j$-invariant takes the following values: 
\begin{align*}
0, 12^3, -15^3, 20^3, -32^3, 2 \cdot (15)^3, 66^3, -96^3, -3 \cdot 160^3, 255^3, -960^3, -5280^3, -640320^3.
\end{align*}
Since $720$ does not appear in this list we have a contradiction, so $\tau_0$ is not a CM point.

	Using the same idea one can show that, among the theta functions associated to the other $23$ Niemeier lattices, only the one corresponding to the root systems $3 E_8$ and $D_{16} \oplus E_8$ (with theta function $E_4^3$) has roots only at CM points and is a Borcherds product; the others have simple roots only at non-CM points.
	
	The above argument can sometimes be applied to theta functions for $\Gamma_0(N)$, using the following idea. For a positive definite, even lattice $L$ of even rank $r$ and level $N$, the associated theta function $\vartheta_L$ is a modular form of weight $k=r/2$ for $\Gamma_0(N)$ with some nebentypus character. Hence the average
	\begin{align}\label{average}
	\prod_{M \in \Gamma_0(N) \backslash \SL_2(\Z)}\vartheta_L \big|_{k}M
	\end{align}
	is a modular form of weight $[\SL_2(\Z): \Gamma_0(N)]\cdot k$ for $\SL_2(\Z)$. The Fourier expansion of $\vartheta_L \big|_{k}M$ can be computed explicitly using the theta transformation formula, see \cite[\S 4.9]{miyake}. If we can show that this average has roots at non-CM points, then the same is true for $\vartheta_L$, so $\vartheta_L$ won't be a Borcherds product.
	
	As an example, we consider the positive definite, binary hermitian form
	\[
	A(x,y) = 2|x|^2 + \tr(x\overline{y}) + 2|y|^2 \qquad (x,y \in \C)
	\]
	over the ring of integers of $\Q(\sqrt{-3})$. The corresponding quaternary quadratic form has level $9$ and discriminant $81$, so its theta function $\vartheta_A = 1 + 18q^2 + 12q^3 + 36q^5 + \dots$ is a modular form of weight $2$ for $\Gamma_0(9)$. The average in \eqref{average} is a non-zero constant multiple of
	\[
	\Delta(\tau)^2 \cdot j(\tau) \cdot \big(j(\tau) + 1167051/512\big).
	\]
	Since $j(\tau)$ is an algebraic integer if $\tau$ is a CM point, the third factor has a root at a non-CM point, so $\vartheta_A$ cannot be a Borcherds product. This example also shows that quaternary theta functions coming from binary hermitian forms over imaginary quadratic fields are not Borcherds products in general, so the case of $\Z[i]$ discussed above seems to be rather special.

	\subsection{Theta functions that are not Fricke eigenforms} Another way to show that a theta function is not a Borcherds product on $\Gamma_0(N)$ is to use the fact that Borcherds products are eigenforms under the Fricke involution $W_N$. Let $L$ be a positive definite even lattice and let $L'$ be its dual lattice. Then, by the theta transformation formula, if $\vartheta_L$ is an eigenform under $W_N$ then the theta functions of $L$ and $L'(N)$ agree\footnote{In general, this does not imply that $L$ is isomorphic to $L'(N)$. Lattices with $L \cong L'(N)$ are called \emph{$N$-modular}, and their theta functions are eigenforms of the Fricke involution $W_N$, compare \cite{quebbemann}.}, where $L'(N) = (L',NQ(x))$ denotes $L'$ rescaled by $N$. This poses some strong restrictions on $L$. For instance, $N$ must be a multiple of the level of $L$ and the discriminant of $L$ must be equal to $N^{r/2}$, where $r$ is the rank of $L$.
	
	For example, using this argument it is easy to see that the theta functions corresponding to the root lattices $A_n$ ($n \geq 4$ even), $D_n$ ($n \geq 6$ even), and $E_6$ cannot be Fricke eigenforms (hence not Borcherds products). Combining this with Example~\ref{example root lattice A2} and Example~\ref{example root lattices}, we see that, among the ADE root lattices of even rank, only $A_2, D_4$, and $E_8$ have theta functions which are Borcherds products.

	\subsection{Other theta functions that are Borcherds products} 
	We have shown that every weight $1$ binary theta function, and certain weight $2$ and weight $4$ theta functions (coming from binary hermitian forms over $\Z[i]$ and $\Z[i,j,k]$) are Borcherds products. However, there are theta functions of weight two that are Borcherds products but cannot be explained by the methods of this paper. For example, consider the four-dimensional lattice with Gram matrix
	\[
	A = \begin{pmatrix}
	2 & -1 & -1 & -1 \\
	-1&  2 & 0  & 0 \\
	-1&  0 & 4 & -1 \\
	-1 & 0 & -1 & 4 
	\end{pmatrix}
	\]
	with discriminant $25$ and level $5$, such that $A \oplus A \cong A_4 \oplus A_4'(5)$. The corresponding theta function $\vartheta_A = 1+6q+18q^2+24q^3+42q^4 + \dots$ has weight $2$ for $\Gamma_0(5)$, and is indeed a Borcherds product, associated with the weight $1/2$ Kohnen plus space form $q^{-4} + 1 - 6q + 3q^4 + 10q^5 + 14q^9 - 60q^{16} + \dots$ for $\Gamma_0(20)$. However, there is no positive definite, integral binary hermitian form over $\Z[i]$ with discriminant $-5$, so this Borcherds product is not explained by Theorem~\ref{theorem product expansion hermitian intro}. One could also ask whether the weight 4 theta series $\vartheta_A^2 = \vartheta_{A_4} \cdot \vartheta_{A_4'(5)}$ can be constructed using a binary quaternionic form, but this is not the case: there is a single orbit of norm $-5$ vectors in the lattice $U+U(4)+3A_1$, and they correspond to the lattice in the genus of $A_4 \oplus A_4'(5)$ (of size 5) with $8$ roots and theta series $1 + 8q + 88q^2 + 256q^3 + 664q^4 + ...$ which has simple zeros in the CM points of discriminant 11. \\
	
	 Classifying the weight $2$ theta functions which are Borcherds products is an interesting open problem.

\bibliographystyle{plain}
\bibliofont
\bibliography{references}

\end{document}